\documentclass[a4paper,12pt]{article} 
\title{Polynomiality of the double ramification cycle}

\usepackage{amsmath, amssymb, mathrsfs, amsthm, shorttoc, geometry, stmaryrd}
\usepackage{tikz}
\usepackage{mathtools} 
\usepackage{hyperref}
\usepackage{xcolor}
\usepackage{subcaption}
\hypersetup{
colorlinks,
linkcolor={black},
citecolor={blue!50!black},
urlcolor={blue!80!black}
}
\usepackage[all]{xy}
\usepackage{tabularx}
\usepackage{longtable}
\numberwithin{equation}{subsection}
\usepackage{enumitem}
\usepackage{mathtools}

\usepackage{cleveref}
\crefformat{equation}{(#2#1#3)}

\AtBeginDocument{\renewcommand{\ref}[1]{\Cref{#1}}}

\usepackage{pgf,tikz}
\usetikzlibrary{matrix, calc, arrows}

\usepackage{tikz-cd} 

\usepackage{todonotes}
\reversemarginpar

\makeatletter
\newcommand*{\doublerightarrow}[2]{\mathrel{
  \settowidth{\@tempdima}{$\scriptstyle#1$}
  \settowidth{\@tempdimb}{$\scriptstyle#2$}
  \ifdim\@tempdimb>\@tempdima \@tempdima=\@tempdimb\fi
  \mathop{\vcenter{
    \offinterlineskip\ialign{\hbox to\dimexpr\@tempdima+1em{##}\cr
    \rightarrowfill\cr\noalign{\kern.5ex}
    \rightarrowfill\cr}}}\limits^{\!#1}_{\!#2}}}
\newcommand*{\triplerightarrow}[1]{\mathrel{
  \settowidth{\@tempdima}{$\scriptstyle#1$}
  \mathop{\vcenter{
    \offinterlineskip\ialign{\hbox to\dimexpr\@tempdima+1em{##}\cr
    \rightarrowfill\cr\noalign{\kern.5ex}
    \rightarrowfill\cr\noalign{\kern.5ex}
    \rightarrowfill\cr}}}\limits^{\!#1}}}
\makeatother

\newcommand{\ca}[1]{{\mathcal{#1}}}

\newcommand{\Mbar}{\overline{\ca M}}

\newcommand{\DRL}{\operatorname{DRL}}
\newcommand{\DR}{\operatorname{DR}}

\DeclareMathOperator{\Aut}{Aut}

\DeclareMathOperator{\CH}{CH}

\newcommand{\G}{\mathbb{G}}

\newcommand{\N}{\mathbb{N}}
\renewcommand{\P}{\mathbb{P}}
\newcommand{\Q}{\mathbb{Q}}

\newcommand{\Z}{\mathbb{Z}}

\newcommand{\Mcal}{\mathcal{M}}

\newcommand{\Ocal}{\mathcal{O}}

\tikzset{
    labl/.style={anchor=south, rotate=-90, inner sep=.5mm}
}

\theoremstyle{definition}
\newtheorem{definition}{Definition}[section]

\newtheorem{remark}[definition]{Remark}

\theoremstyle{plain}

\newtheorem{proposition}[definition]{Proposition}
\newtheorem{lemma}[definition]{Lemma}
\newtheorem{theorem}[definition]{Theorem}
\newtheorem{corollary}[definition]{Corollary}

\newtheorem{intheorem}{Theorem}

\theoremstyle{remark}

\usepackage{letltxmacro}
\LetLtxMacro{\phiorig}{\phi}
\renewcommand{\phi}{\varphi}

\DeclarePairedDelimiter\floor{\lfloor}{\rfloor}

\author{Pim Spelier}
\date{\today}

\newcounter{nootje}
\newcommand{\beq}{\begin{equation}}
\newcommand{\eeq}{\end{equation}}
\newcommand{\beqs}{\begin{equation*}}
\newcommand{\eeqs}{\end{equation*}}

\tikzset{
  symbol/.style={
    draw=none,
    every to/.append style={
      edge node={node [sloped, allow upside down, auto=false]{$#1$}}}
  }
}

\begin{document}
\maketitle

\begin{abstract}
Let $A = (a_1,\dots,a_n)\in \Z^n$ be a sequence with sum $k(2g-2+n)$. The double ramification cycle $\DR_g(A) \in \CH^g(\Mbar_{g,n})$ is the virtual class of the locus of curves $(C,p_1,\dots,p_n)$ where the line bundle $(\omega_C^{\log})^{-k}\left(\sum a_i p_i\right)$ is trivial. Although there has long been a formula for $\DR_g(A)$ \cite{Janda2016Double-ramifica}, the exact dependence on $A$ was unknown for a long time, though it was conjectured to be polynomial in $A$. A proof was announced in \cite{Janda2016Double-ramifica}, and Pixton gave a proof incorporating ideas of Zagier in \cite{Pixton2023DRPoly}. Here we present an alternative proof of the polynomiality of the double ramification cycle.
\end{abstract}


\section{Introduction}
\label{sec:intro}
Let $\Mcal_{g,n}$ be the moduli space of smooth curves $(C/S,p_1,\dots,p_n)$ of genus $g$ with $n$ distinct markings. Let $A = (a_1,\dots,a_n)\in \Z^n$ be a sequence with sum $0$. Then the \emph{double ramification locus} is
\[
\DRL_g(A) = \left\{(C/S, p_1,\dots,p_n) : \Ocal_C\left(\sum a_i p_i\right) \text{ is fiberwise trivial}\right\} \subset \Mcal_{g,n}.
\]
Equivalently, this is the locus of curves $C$ with a rational function $f: C \to \P^1$ with specified ramification profile above $0$ and $\infty$.

In 2001 Eliashberg asked the question of how to compactify this substack, and how to compute the compactification. Compactifying the substack was first done in \cite{Graber2003Hodge-integrals}, using stable maps to $\P^1$ modulo the $\G_m$ action. They defined the \emph{double ramification cycle} \[\DR_g(A) \in \CH^g(\Mbar_{g,n})\] as the virtual class of this compactification. There are other equivalent definitions, using birational geometry of $\Mbar_{g,n}$ \cite{Holmes2017Extending-the-d} or logarithmic geometry \cite{Marcus2017Logarithmic-com}. These latter definitions also generalise to the \emph{twisted double ramification cycle}
\[
\DR_g(A)
\]
for a sequence of integers $A \in \Z^n$ summing to $k(2g-2+n)$ for some $k \in \Z$. This double ramification cycle is the virtual class of the compactification of the locus
\[
\left\{(C/S, p_1,\dots,p_n) : (\omega_C^{\log})^{-k}\left(\sum a_i p_i\right) \text{ is fiberwise trivial}\right\}.
\]
The double ramification cycle has connections with ordinary Gromov--Witten invariants through the localisation formula \cite{Janda2016Double-ramifica,Ranganathan2023Logarithmic-Gromov-Witten}. It also has ties with the world of PDE's \cite{Buryak2015Double-ramifica,Buryak2019Quadratic-doubl} and to gauge theory \cite{Frenkel2009Gromov-Witten-g}. It has also been used to find relations in the tautological ring \cite{Clader2018PixtonFormulaRelations}.

Pixton conjectured a formula for the double ramification cycle, and in 2014 this was proven in \cite{Janda2016Double-ramifica} for the case $k = 0$, thereby answering the second part of Eliashberg's question. They define for every integer $r \in \Z_{\geq 1}$ an explicit class
\[
P_g^r(A) \in \CH^g(\Mbar_{g,n}),
\]
in terms of decorated strata multiplied by certain numbers associated to the combinatorics of the graph and the integers $a_1,\dots,a_n$ and $r$. 

They prove that $P_g^r(A)$ is polynomial in $r$ for $r$ large enough. Then they define $P_g^0(A)$ to be the constant term of this polynomial expression, and they prove the equality \[\DR_g(A) = P_g^0(A).\] In \cite{Bae2020Pixtons-formula} this result was extended to a formula for the twisted double ramification cycle.

These two results were major breakthroughs, but a very simple questions remained open: what is the behaviour of $\DR_g(A)$ in terms of $A$? In \cite{Janda2016Double-ramifica} the double ramification cycle was conjectured to be polynomial in $A$, but this is a deceptively difficult question, as the formula itself has no obvious polynomial dependency on $A$. A proof was announced in \cite{Janda2016Double-ramifica}. A proof incorporating ideas of Zagier was made public in 2023 \cite{Pixton2023DRPoly}. In this paper, we present a different proof of the same theorem.

\begin{intheorem}[\Cref{thm:drpoly}]
\label{inthm:drpoly}
Fix $g,n$. The cycle $\DR_g(a_1,\dots,a_n) \in \CH^g(M_{g,n})$ is a polynomial in $(a_1,\dots,a_n) \in \Z^n$, where we require that $(2g-2+n)\mid \sum a_i$.
\end{intheorem}

We use different techniques to prove this statement. Our methods are explicit and can be recursively used to give a polynomial expression for the double ramification cycle.

In \ref{sec:drdefinitions} we recall the combinatorics necessary to state Pixton's formula. In \ref{sec:pixton} we give a brief recap of Pixton's formula. In \ref{sec:sums} we prove the main technical result, a certain polynomiality statement for sums of weightings on graphs. In \ref{sec:drpoly} we finally give the proof of \ref{inthm:drpoly}.

\section{Definitions}
\label{sec:drdefinitions}

\begin{definition}
We denote by $G_{g,n}$ the set of graphs
\[
\Gamma = (V,H,L = (\ell_i)_{i=1}^n, r: H \to V, i: H \to H, g: V \to \N)
\]
where
\begin{enumerate}
  \item $V$ is the set of vertices;
  \item $H$ is the set of half edges;
  \item $i$ is an involution on $H$.
  \item $r$ assigns to every half edge the vertex it is incident to.
  \item $L = (\ell_i)_{i=1}^n \subset H$ is a list of the legs, the fixed points of the involution $i$;
  \item $(V,H)$ is a connected graph;
  \item $g(v)$ is the genus of vertex $v$;
  \item for every vertex $v$ we have $2g(v)-2 + n(v) > 0$ where $n(v)$ is the number of half edges incident to $v$;
  \item the genus $\sum_{v \in V} g(v) + h^1(\Gamma)$ of the graph is $g$.
\end{enumerate}
For such a graph, we denote by $E$ the set of edges, i.e. the set of pairs $\{h,h'\}$ of half edges with $h = i(h')$ and $h \neq h'$.
\end{definition}

Then $G_{g,n}$ parametrises the strata of $\Mbar_{g,n}$. Every graph $\Gamma \in G_{g,n}$ corresponds to a moduli space $\Mbar_{\Gamma} = \prod_{v \in V(\Gamma)} \Mbar_{g(v),n(v)}$ together with a glueing map $\zeta_{\Gamma}: \Mbar_{\Gamma} \to \Mbar_{g,n}$. 

\begin{definition}
We fix a vector $A \in \Z^n$ with sum $k(2g-2+n)$. A \emph{weighting for $A$} on a graph $\Gamma \in G_{g,n}$ is a map $w: H(\Gamma) \to \Z$ satisfying the following three conditions.
\begin{enumerate}
  \item For $i = 1, \dots, n$ we have $w(\ell_i) = a_i$.
  \item For $h \in H \setminus L$ we have $w(h) + w(i(h)) = 0$.
  \item For every vertex $v$ we have $\sum_{h \colon r(h) = v} w(h) = k(2g(v) - 2 + n(v))$.
\end{enumerate} 
For $r \in \N_{\geq 1}$, an \emph{weighting for $A$} modulo $r$ is a map $w: H(\Gamma) \to \{0,\dots,r-1\} \subset \Z$ satisfying these three conditions modulo $r$. We denote $W^{\Gamma}_{A,r}$ for the finite set of weightings for $A$ mod $r$.
\end{definition}

We fix $g,n$, and we fix a graph $\Gamma \in G_{g,n}$. Let $Q \in \Z[x_{h} : h \in H \setminus L]$ be a polynomial.

\begin{definition}
\label{def:thesum}
For $A = (a_1,\dots,a_n)$ with $\sum a_i = k(2g-2+n)$, consider the sum
\begin{equation}
\label{eq:thesum}
S_{A,r}^\Gamma(Q) = r^{-h_1(\Gamma)} \sum_{w \in W^{\Gamma}_{A,r} } Q(w(h)).
\end{equation}
\end{definition}

In \cite[Appendix~A]{Janda2018Double-ramifica} they prove the following property of these sums.
\begin{proposition}[\protect{\cite[Proposition~3'']{Janda2016Double-ramifica}}]
\label{prop:eventuallypolynomial}
The sum $S_{A,r}^{\Gamma}(Q)$ are eventually polynomial in $r$.
\end{proposition}

This allows them to define the following quantity.
\begin{definition}
\label{def:constantterm}
We denote the constant term of the polynomial expression for $S_{A,r}^{\Gamma}(Q)$ for large $r$ by
\[
S_{A,0}^\Gamma(Q).
\]
\end{definition}

\section{Pixton's formula for the double ramification cycle}
\label{sec:pixton}

Fix $g,n \in \N,k \in \Z$ with $2g - 2 + n > 0$. Fix a sequence $A \in \Z^n$ with sum $k(2g-2+n)$. In this section we will present Pixton's formula for the double ramification cycle $\DR_g(A) \in \CH^*(\Mbar_{g,n})$. The first version of this, valid for $k = 0$, was \cite[Theorem~1]{Janda2016Double-ramifica}. We will present the version from \cite[Theorem~7]{Bae2020Pixtons-formula}. 

For $r \in \N_{\geq 1}$, we define the term $P_r^\Gamma \in \CH(\Mbar_{\Gamma})$ as follows.
\[
P_r^{\Gamma} = r^{-h^1(\Gamma)}\sum_{w \in W_{A,r}^\Gamma} \prod_{e = \{h,h'\} \in E(\Gamma)} \frac{1-\exp\left(\frac{w(h)w(h')}{2} (\psi_h + \psi_{h'})\right)}{\psi_h + \psi_{h'}}.
\]
We remark that the factor
\[
\frac{1-\exp\left(\frac{w(h)w(h')}{2} (\psi_h + \psi_{h'})\right)}{\psi_h + \psi_{h'}}
\]
is well-defined, as the denominator formally divides the numerator.

For every choice of monomial in the factors $(\psi_h + \psi_h')$ for edges $\{h,h'\}$, the corresponding coefficient of $P_r^\Gamma$ is of the form $S_{A,r}^\Gamma(Q)$ for some polynomial $Q$ where $S_{A,r}^\Gamma(Q)$ is as defined in \ref{def:thesum}. Thus $P_r^\Gamma$ is eventually polynomial per \ref{prop:eventuallypolynomial}. We define
\[
P_0^\Gamma \in \CH(\Mbar_{\Gamma})
\]
to be the element obtained by taking the constant term of the polynomial expression.

Then by \cite[Theorem~7]{Bae2020Pixtons-formula} we have the formula
\begin{equation}
\label{eq:pixton}
\DR_g(A) = \left[\exp\left(-\frac12 (k^2 \kappa_1 - \sum_i a_i^2 \psi_i)\right) \sum_{\Gamma \in G_{g,n}} \frac{1}{|\Aut(\Gamma)|}\zeta_{\Gamma,*}P_0^\Gamma \right]^g,
\end{equation}
where $[\cdot]^g$ means the codimension $g$ term.


\section{Sums over weightings}
\label{sec:sums}
We fix $g,n$ as in the introduction, and we fix a graph $\Gamma \in G_{g,n}$. Write $H \setminus L = \{h_1, \dots, h_m\}$. Let $Q \in \Z[x_h : h \in H \setminus L]$ be a polynomial. 

Fix $A\in \Z^n$ with $\sum a_i = 0$. We recall the sum
\[
S_{A,r}^\Gamma(Q) = r^{-h_1(\Gamma)} \sum_{w \in W^{\Gamma}_{A,r} } Q(w(h))
\]
from \ref{eq:thesum}, and the constant term of the polynomial expression for the sum for large $r$
\[
S_{A,0}^\Gamma(Q).
\]

The main theorem in this section is the following polynomiality statement.

\begin{theorem}
\label{thm:sumspolynomials}
For every polynomial $Q \in \Z[x_h : h \in H \setminus L]$ the function \begin{align*}\{A \in \Z^n : \sum a_i = 0\} &\to \Q \\ A &\mapsto S_{A,0}^\Gamma(Q)\end{align*} is polynomial in $A$.
\end{theorem}

We will prove this by induction on the number of edges of $\Gamma$. In \Cref{subsec:separating} we treat the case where $\Gamma$ has a separating edge. In \Cref{subsec:non-separating} we prove a preliminary result in the case where $\Gamma$ has a non-separating edge. In \Cref{subsec:keqzero} we put everything together and finish the proof of \Cref{thm:sumspolynomials}. We will then prove \ref{cor:sumspolynomials}, generalising \ref{thm:sumspolynomials} to the domain $\{A \in \Z^n : (2g-2+n)\mid \sum a_i\}$.

We first make some general observations and definitions.

\begin{remark}
We will use the notation $v \bmod r$ to denote the unique element in $\{0,\dots,r-1\}$ that is modulo $r$ equivalent to $v$. We use the notation $v \equiv w \pmod r$ to denote that $n$ and $m$ are equivalent modulo $r$.
\end{remark}

\begin{lemma}
Fix $A$. Then for $r$ large enough and coprime to a fixed integer, we have
\[
S_{A,0}^\Gamma(Q) \equiv S_{A,r}^\Gamma(Q) \pmod r.
\]
\end{lemma}
\begin{proof}
We know $S_{A,r}^\Gamma(Q)$ is for $r$ large enough a polynomial in $r$, with rational coefficients. If the product of the denominators is coprime to $r$, this means that $S_{A,r}^\Gamma(Q) \equiv S_{A,0}^\Gamma(Q) \pmod r$.
\end{proof}

\begin{definition}
\label{def:graphcut}
Let $e = \{h_1,h_2\}$ be an edge of $\Gamma$. We let $\Gamma_e$ denote the graph $\Gamma$ with the edge $e$ removed, and with two legs $\ell_{n+1}, \ell_{n+2}$ added at the roots of the half-edges $r(h_1), r(h_2)$.
\end{definition}

\subsection{Separating edge case}
\label{subsec:separating}
In this section we will prove \Cref{thm:sumspolynomials} in the case that $\Gamma$ has a separating edge.

\begin{proposition}
\label{prop:separating}
Assume $\Gamma$ has a separating edge. Then \ref{thm:sumspolynomials} holds, assuming that it holds for graphs with fewer edges than $\Gamma$.
\end{proposition}
\begin{proof}
By linearity we can assume $Q$ is a monomial, and we write \[Q = \prod_{h \in H \setminus L} x_h^{c_h}.\]

Let $e = \{h_1,h_2\}$ be a separating edge of $\Gamma$.

Then $\Gamma_e$ is a disjoint union of two graphs, denoted $\Gamma_1 \sqcup \Gamma_2$. Let $A_1$ denote the subsequence of $A$ consisting of elements whose corresponding vertex lies in $\Gamma_1$, and similarly for $A_2$. We denote the sums $\sum_{a \in A_i} a$ by $s_i$ for $i \in \{1,2\}$. For $i \in \{1,2\}$ let $Q_i \in \Z[x_h \colon h \in H_i \setminus L_i]$ denote the monomial $Q$ with all half edges not in $H_i \setminus L_i$ set to $1$, and let $c_i$ denote $c_{h_i}$.

Note that then $W^{\Gamma}_{A,r}$ splits as $W^{\Gamma_1}_{(A_1,-s_1),r} \times W^{\Gamma_2}_{(A_2,-s_2),r}$. Denote these summands by $W_1$ and $W_2$ respectively. Then we find the formula
\begin{align*}
S_{A,r}^\Gamma(Q) &= r^{-h_1(\Gamma)} \sum_{w_1 \in W_1} \sum_{w_2 \in W_2} Q_1(w_1) (-s_1 \bmod r)^{c_1} (-s_2 \bmod r)^{c_2} Q_2(w_2)\\
               &= \left(r^{-h_1(\Gamma_1)} \sum_{w_1 \in W_1} Q_1(w_1)\right) \cdot  (-s_1 \bmod r)^{c_1} \cdot \\ & {\hspace{173pt}} (-s_2 \bmod r)^{c_2} \cdot \left(r^{-h_1(\Gamma_2)} \sum_{w_2 \in W_2} Q_2(w_2)\right)\\ 
               &=S_{(A_1,-s_1),r}^{\Gamma_1}(Q_1) \cdot (-s_1 \bmod r)^{c_1} (-s_2 \bmod r)^{c_2} \cdot S_{(A_2,-s_2),r}^{\Gamma_2}(Q_2)
\end{align*}
By our assumption that \ref{thm:sumspolynomials} holds for graphs with fewer edges, we know the first and last factors are modulo $r$ equal to $S_{(A_i,-s_i),0}^{\Gamma_1}(Q_i)$ with $i=1$ and $2$ respectively (for $r$ large enough and coprime to a fixed integer). In total we see that for such $r$ we have the following equalities
\begin{align*}
S_{A,0}^\Gamma(Q) &\equiv S_{A,r}^\Gamma(Q) &&\hspace{-8pt}\pmod r \\
               &= S_{(A_1,-s_1),r}^{\Gamma_1}(Q_1) \cdot (-s_1 \bmod r)^{c_1} (-s_2 \bmod r)^{c_2} \cdot S_{(A_2,-s_2),r}^{\Gamma_2}(Q_2)&\\
               &\equiv S_{(A_1,-s_1),0}^{\Gamma_1}(Q_1) \cdot (-s_1)^{c_1}(-s_2)^{c_2} \cdot S_{(A_2,-s_2),0}^{\Gamma_2}(Q_2) &&\hspace{-8pt}\pmod r.\\
\end{align*}
By our assumption that \ref{thm:sumspolynomials} holds for graphs with fewer edges than $\Gamma$ the term \[S_{(A_1,-s_1),0}^{\Gamma_1}(Q_1) \cdot (-s_1)^{c_1}(-s_2)^{c_2} \cdot S_{(A_2,-s_2),0}^{\Gamma_2}(Q_2)\] is polynomial in $A$. As it agrees with $S_{A,0}^{\Gamma}$ modulo an infinite number of integers, we have
\begin{equation}
S_{A,0}^{\Gamma}(Q) = S_{(A_1,-s_1),0}^{\Gamma_1}(Q_1) \cdot (s_1 s_2)^{c_1} \cdot S_{(A_2,-s_2),0}^{\Gamma_2}(Q_2)
\end{equation}
and in particular, $S_{A,0}^{\Gamma}(Q)$ is polynomial in $A$.
\end{proof}

\subsection{Non-separating case}
\label{subsec:non-separating}
In this section we will treat the case where $\Gamma$ has a non-separating edge.

\begin{proposition}
\label{prop:nonseparating}
Let $e = \{h_1,h_2\}$ be a non-separating edge, and denote $v_i = r(h_i)$ the root of $h_i$ for $i \in \{1,2\}$. Assume for $i \in \{1,2\}$ that there is a leg $\ell_i$ with insertions $a_i$ and adjacent to $v_i$. Let $A \in \Z^n$ be a vector with total sum zero. For $a \in \Z$, let $A_a$ be the vector $(a_1-a,a_2+a,a_3,a_4,\dots,a_n)$. Assume that \ref{thm:sumspolynomials} holds for graphs with fewer edges than $\Gamma$.

Then there exists a polynomial $\Psi_{\Gamma,e,Q}(x_1,\dots,x_n,y,z_1,\dots,z_m)$ and polynomials $R_1,\dots,R_m \in \Z[x_h : h \in H \setminus L]$ such that
\[
S_{A_a,0}^\Gamma(Q) = \Psi_{\Gamma,e,Q}(a_1,\dots,a_n,a,S_{A,0}^\Gamma(R_1),\dots,S_{A,0}^\Gamma(R_m)).
\]
\end{proposition}
\begin{proof}
By linearity we can assume $Q$ is a monomial, and we write \[Q = \prod_{h \in H \setminus L} x_h^{c_h}.\] For $i \in \{1,2\}$ we let $c_i$ denote $c_{h_i}$.

We first prove the case where $a \geq 0$. We fix an $r$, and let $W_a$ denote $W_{A_a,r}^{\Gamma}$. We will first rewrite $S_{A_a,r}^{\Gamma}(Q)$. Note there is a bijection
\begin{align*}
\phi_a \colon W_0 &\to W_a\\
              w & \mapsto \left(h_i \mapsto\begin{cases}(w(h_1) + a)\bmod r & \text{if $i$ = 1} \\ (w(h_2) - a) \bmod r & \text{if $i$ = 2}\\ w(h_i) & \text{if } i > 2\end{cases}\right).
\end{align*}

Let $Q_0(w)$ denote the monomial $Q$ with $x_{h_1}$ and $x_{h_2}$ substituted by $1$. We then have the formula
\begin{align*}
S_{A_a,r}^\Gamma(Q) &= r^{-h_1(\Gamma)} \sum_{w \in W_0} Q(\phi_a(w)) \\
                 &= r^{-h_1(\Gamma)} \sum_{w \in W_0} ((w(h_1) + a) \bmod r)^{c_1} ((w(h_2) - a) \bmod r)^{c_2} Q_0(w).
\end{align*}
Now we will rewrite the factor $((w(h_1) + a) \bmod r)^{c_1} ((w(h_2) - a) \bmod r)^{c_2}$, using some facts about $m \bmod r$ for integers $m$.

For any $m \not\equiv 0 \pmod r$ we have $(r-m) \bmod r = r - (m \bmod r)$ and hence \[(m \bmod r)^{c_1} ((r-m) \bmod r)^{c_2} = (m \bmod r)^{c_1} (r - (m \bmod r))^{c_2}.\] For $m \equiv 0 \pmod r$ we have \[(m \bmod r)^{c_1} ((r-m) \bmod r)^{c_2} = (m \bmod r)^{c_1} (r - (m \bmod r))^{c_2} - (1-\delta_{c_2,0}) 0^{c_1} r^{c_2}\] where $\delta$ is the Kronecker delta. (This term $(1-\delta_{c_2,0})0^{c_1} r^{c_2}$ is usually $0$, but can be non-zero in the case $c_1 = 0 < c_2$).
Also, $m \bmod r = m - r \floor{\frac{m}{r}}$ is a polynomial in $m$ and $r \floor{\frac{m}{r}}$.

Now we will use this for $m = w(h_1)+a$. All in all, the factor $((w(h_1) + a) \bmod r)^{c_1} ((w(h_2) - a) \bmod r)^{c_2}$ can be rewritten as \[p\left(w(h_1),a,r,r\floor{\frac{w(h_1)+a}{r}}\right) - \delta_{w(h_1)+a \bmod r,0}\cdot(1-\delta_{c_2,0})\cdot 0^{c_1} r^{c_2}\] for the polynomial
\[
p(x_0,x_1,x_2,x_3) = \left(x_0+x_1 - x_3\right)^{c_1} \cdot \left(x_2- x_0 - x_1 + x_3\right)^{c_2}.
\]
Now we assume that $r > a$. Then we have $0 \leq w(h_1) + a < 2r$ and so

\begin{equation}
\label{eq:floor}
\floor{\frac{w(h_1)+a}{r}} = \begin{cases}0 &\text{if } r- w(h_1) > a \\ 1 &\text{if } r-w(h_1) \leq a\end{cases}.
\end{equation}
Define $p_0(x_0,x_1,x_2) = p(x_0,x_1,x_2,0)$ and $p_1(x_0,x_1,x_2,x_3) = \frac{p-p_0}{x_3}$. Then we can write \[p = p_0(x_0,x_1,x_2) + x_3 p_1(x_0,x_1,x_2,x_3).\]

Write
\[
S_1 \coloneqq r^{-h_1(\Gamma)} \sum_{w \in W_0} p_0(w(h_1),a,r) Q_0(w)
\]
and
\[
S_2 \coloneqq r^{-h_1(\Gamma_e)} \sum_{w \in W_0} \floor{\frac{w(h_1)+a}{r}}p_1(w(h_1),a,r,r\floor{\frac{w(h_1)+a}{r}}) Q_0(w)
\]
and
\[
S_3 \coloneqq r^{-h_1(\Gamma)} \sum_{w \in W_0 \colon w(h_1) = -a \pmod r} (1-\delta_{c_2,0})\cdot 0^{c_1} r^{c_2} Q_0(w).
\]
so that $S_{A_a,r}^\Gamma(Q) = S_1 + S_2 - S_3$. Here $\Gamma_e$ is the graph defined in \ref{def:graphcut} we have used that $r^{-h_1(\Gamma)}\cdot r = r^{-h_1(\Gamma_e)}$. We will show that each of $S_1,S_2,S_3$ is modulo $r$ of the required form.

Write $p_0(x_0,x_1,x_2) = \sum_{i=0}^d p_{0,i}(x_0) x_1^i + x_2 p_{2}(x_0,x_1,x_2)$ for some polynomials $p_{0,i},p_2$. Then we get $S_1 \equiv \sum_{i=0}^d a^i S_{A,r}^\Gamma(p_{0,i}(w(h_1)) \cdot Q_0) \pmod r$. Hence for $r$ large enough and coprime to a finite set of integers $S_1$ is modulo $r$ equivalent to a polynomial in $a$ and terms $S_{A,0}^{\Gamma}(R)$ for some polynomials $R$, and in particular modulo $r$ it is of the required form.

By \ref{eq:floor} we can rewrite the sum $S_2$ as
\begin{align*}
S_2 &= r^{-h_1(\Gamma_e)} \sum_{w \in W_0 \colon 1 \leq r-w(h_1) \leq a} p_1(w(h_1),a,r,r) Q_0(w)\\
    &= \sum_{j = 1}^{a} \left( p_1(r-j,a,r,r) r^{-h_1(\Gamma_e)} \sum_{w \in W_{(A,r-j,j-r),r}^{\Gamma_e}} Q_0(w) \right)\\
    &= \sum_{j = 1}^{a} p_1(r-j,a,r,r) S_{(A,-j,j),r}^{\Gamma_e}(Q_0)
\end{align*}
Taking $r$ large enough with respect to $a$ and coprime to a finite set of integers, for every $j$ the factor $S_{(A,-j,j),r}^{\Gamma_e}(Q_0)$ is modulo $r$ equivalent to $S_{(A,-j,j),0}^{\Gamma_e}(Q_0)$. By our induction hypothesis applied to the graph $\Gamma_e$, the term $q(A,j) \coloneqq S_{(A,-j,j),0}^{\Gamma_e}(Q_0)$ is a polynomial in $A,j$.
Then $S_2 \equiv \sum_{j=1}^a p_1(-j,a,0,0) q(A,j) \pmod r$, and $\sum_{j=1}^a p_1(-j,a,0,0) q(A,j)$ is a polynomial in $A$ and $a$ for $a \geq 0$.

Finally, $S_3$ is $0$ unless $c_1 = 0 < c_2$, so assume $c_1 = 0 < c_2$. Then
\[
S_3 = r^{c_2-1} S_{(A,-a,a),r}^{\Gamma_e}(Q_0).
\]
Modulo $r$ this is either $0$, or $S_{(A,-a,a),0}^{\Gamma_e}(Q_0)$ which by the induction hypothesis is polynomial in $A,a$.

In total $S_{A_a,r}^\Gamma(Q)$ is modulo $r$ equivalent to a fixed polynomial $\Psi_{\Gamma,e,Q}$ in $A,a$ and in terms $S_{A,0}^\Gamma(R_i)$ for polynomials $R_1,\dots,R_m$, for every $r$ large enough and coprime to a fixed integer. This means that for $a \geq 0$ we find \[S_{A_a,0}^\Gamma(Q) = \Psi_{\Gamma,e,Q}(a_1,\dots,a_n,a,S_{A,0}^\Gamma(R_1),\dots,S_{A,0}^\Gamma(R_m)).\]

This finishes the proof for $a \geq 0$.

Now we apply this to the edge with the opposite orientation and to $A_a$ with $a$ positive. We find that for $a,b \geq 0$ we have that $S_{A_{a-b},0}^{\Gamma}(Q)$ is a polynomial in $A,a,b$ and terms $S_{A_a,0}^\Gamma(R)$ for polynomials $R$. These latter are polynomials in $A,a,S_{A,0}^\Gamma(R')$ for polynomials $R'$, so in total
\[
S_{A_{a-b},0}^{\Gamma}(Q)
\]
is for $a,b \geq 0$ a polynomial in $A,a,b$ and terms $S_{A,0}^\Gamma(R)$. As it is also a function of $A,a-b$ and terms $S_{A,0}^\Gamma(R)$, it is a polynomial in $A,a-b$ and terms $S_{A,0}^\Gamma(R)$ and we find that
\[S_{A_a,0}^\Gamma(Q) = \Psi_{\Gamma,e,Q}(a_1,\dots,a_n,a,S_{A,0}^\Gamma(R_1),\dots,S_{A,0}^\Gamma(R_m))\]
holds for any $a \in \Z$.
\end{proof}

\subsection{\texorpdfstring{Proof of \Cref{thm:sumspolynomials}}{Proof of Theorem 4.1}}
\label{subsec:keqzero}
Now we have assembled the ingredients to prove \ref{thm:sumspolynomials}.

\begin{proof}[Proof of \Cref{thm:sumspolynomials}]
We will prove this by induction on the number of edges of $\Gamma$. For $\Gamma$ a graph without edges this is immediately clear. Now we assume \ref{thm:sumspolynomials} holds for graphs with fewer edges than $\Gamma$.

If $\Gamma$ has a separating edge, it follows immediately from \Cref{prop:separating}. We now assume there are no separating edges.

We can assume there is exactly one leg at every vertex. We number the vertices $1$ through $n$, in accordance with the attached leg. Let $T$ be a spanning tree of $\Gamma$. Without loss of generality, our numbering is such that for every $k = 1,\dots,n$ the vertices $\{1,\dots,k\}$ form a subtree $T_k$ of $T$. We denote \[P_k \coloneqq \{A \in \Z^n \colon \sum_{i=1}^k a_i = 0 \wedge a_{k+1} = \dots = a_n = 0 \}.\] 

Then we will prove with induction on $k$ that for every polynomial $Q$ the map $A \mapsto S_{A,0}^{\Gamma}(Q)$ is polynomial when restricted to $P_k$.

We start with a subtree with one vertex. By the total sum zero condition we have $P_1 = \{0\}$, and then for every polynomial $Q$ the map $P_k \to \Z, A \mapsto S_{A,0}^{\Gamma}(Q)$ is a constant map, and in particular polynomial.

Now assume we have proven it for $k$ and want to prove it for $k+1$. Let $j \leq k$ be the vertex that in $T$ is adjacent to vertex $k+1$. Note that $P_{k+1} = P_k + \{a (e_k - e_j) \colon a \in \Z\}$ where $e_i$ is the $i$'th basis vector in $\Z^n$. The induction step follows from \Cref{prop:nonseparating}. As $P_n = \{A \in \Z^n \colon \sum_i a_i = 0\}$ we are done.
\end{proof}

We immediately find the following corollary.
\begin{corollary}
\label{cor:sumspolynomials}
The function \begin{align*}\{A \in \Z^n : (2g-2+n)\mid\sum a_i\} &\to \Q \\ A &\mapsto S_{A,0}^\Gamma(Q)\end{align*} is polynomial in $A$.
\end{corollary}
\begin{proof}
We will prove this  by reducing to the $k = 0$ case.

Let $\Gamma' \in G_{g,n+\#V}$ be the graph $\Gamma$ with one leg added at every vertex. Let $A'$ denote the vector $(a_1,\dots,a_n,(-k(2g(v)-2+n(v)))_{v \in v})$. We see that the sum of the elements in $A'$ is $0$. Note that $S_{A,r}^\Gamma(Q) = S_{A',r}^{\Gamma'}(Q)$. The theorem follows by applying \ref{thm:sumspolynomials} to $S_{A',r}^{\Gamma'}(Q)$.
\end{proof}

\section{Polynomiality of the double ramification cycle}
\label{sec:drpoly}
In this section we will prove polynomiality of the double ramification cycle. With \Cref{thm:sumspolynomials} and Pixton's formula \ref{eq:pixton}, we conclude with the following theorem.

\begin{theorem}
\label{thm:drpoly}
Fix $g,n$. The cycle $\DR_g(a_1,\dots,a_n) \in \CH^g(M_{g,n})$ is a polynomial in $(a_1,\dots,a_n) \in \Z^n$, where we require that $\sum a_i$ is divisible by $(2g-2+n)$.
\end{theorem}
\begin{proof}
By the formula \ref{eq:pixton}, we have that $\DR_g(a)$ is a polynomial in a finite set of decorated strata, the $a_i$, and in terms $S_{A,0}^{\Gamma}(Q)$. The result follows from \ref{cor:sumspolynomials}.
\end{proof}

\bibliographystyle{hep}
\addcontentsline{toc}{section}{References}
\bibliography{references.bib}


\end{document}